\documentclass[12pt,reqno]{amsart}

\usepackage{times}
\usepackage[T1]{fontenc}
\usepackage{mathrsfs}
\usepackage{latexsym}
\usepackage[dvips]{graphics}
\usepackage{epsfig}
\usepackage{amsmath,amsfonts,amsthm,amssymb,amscd}
\input amssym.def
\input amssym.tex
\usepackage{color}
\usepackage{hyperref}
\usepackage{url}

\newcommand{\burl}[1]{\textcolor{blue}{\url{#1}}}

\newcommand{\emaillink}[1]{\textcolor{blue}{\href{mailto:#1}{#1}}}

\title{Generalizing Zeckendorf's Theorem to $f$-decompositions}

\author[Demontigny]{Philippe Demontigny}
\email{\emaillink{Philippe.P.Demontigny@williams.edu}}
\address{Department of Mathematics \& Statistics, Williams College, Williamstown, MA 01267}

\author[Do]{Thao Do}
\email{\emaillink{thao.do@stonybrook.edu}}
\address{Mathematics Department, Stony Brook University, Stony Brook, NY, 11794}

\author[Kulkarni]{Archit Kulkarni}
\email{\emaillink{auk@andrew.cmu.edu}}
\address{Department of Mathematical Sciences, Carnegie Mellon University, Pittsburgh, PA 15213}

\author[Miller]{Steven J. Miller}
\email{\emaillink{sjm1@williams.edu}, \emaillink{Steven.Miller.MC.96@aya.yale.edu}}
\address{Department of Mathematics and Statistics, Williams College, Williamstown, MA 01267}

\author[Moon]{David Moon}
\email{\emaillink{Dong.Hwan.Moon@williams.edu}}
\address{Department of Mathematics \& Statistics, Williams College, Williamstown, MA 01267}

\author[Varma]{Umang Varma}
\email{\emaillink{Umang.Varma10@kzoo.edu}}
\address{Department of Mathematics \& Computer Science, Kalamazoo College, Kalamazoo, MI, 49006}



\thanks{This research was conducted as part of the 2013 SMALL REU program at Williams College and was partially supported funded by NSF grant DMS0850577 and Williams College; the fourth named author was also partially supported by NSF grant DMS1265673. We would like to thank our  colleagues from the Williams College 2013 SMALL REU program for helpful discussions, especially Francisc Bozgan, Taylor Corcoran, Joseph R Iafrate, Jaclyn Porfilio and Jirapat Samranvedhya, as well as Kevin O'Bryant for helpful conversations.}

\usepackage[margin=1in]{geometry}
\usepackage{mathtools}

\newtheorem{thm}{Theorem}[section]

\newtheorem{cor}[thm]{Corollary}
\newtheorem{lem}[thm]{Lemma}
\newtheorem{prop}[thm]{Proposition}
\newtheorem{exa}[thm]{Example}
\newtheorem{defi}[thm]{Definition}

\numberwithin{equation}{section}

\newcommand{\N}{\mathbb{N}}
\newcommand{\Q}{\mathbb{Q}}
\newcommand{\R}{\mathbb{R}}
\newcommand{\E}{\mathbb{E}}
\newcommand{\stir}[2]{\genfrac{[}{]}{0pt}{}{#1}{#2}} 

\begin{document}
\begin{abstract}
A beautiful theorem of Zeckendorf states that every positive integer can be uniquely decomposed as a sum of non-consecutive Fibonacci numbers $\{F_n\}$, where $F_1 = 1$, $F_2 = 2$ and $F_{n+1} = F_n + F_{n-1}$. For general recurrences $\{G_n\}$ with non-negative coefficients, there is a notion of a legal decomposition which again leads to a unique representation, and the number of summands in the representations of uniformly randomly chosen $m \in [G_n, G_{n+1})$ converges to a normal distribution as $n \to \infty$.

We consider the converse question: given a notion of legal decomposition, is it possible to construct a sequence $\{a_n\}$ such that every positive integer can be decomposed as a sum of terms from the sequence? We encode a notion of legal decomposition as a function $f:\N_0\to\N_0$ and say that  if $a_n$ is in an ``$f$-decomposition'', then the decomposition cannot contain the $f(n)$ terms immediately before $a_n$ in the sequence; special choices of $f$ yield many well known decompositions (including base-$b$, Zeckendorf and factorial). We prove that for any $f:\N_0\to\N_0$, there exists a sequence $\{a_n\}_{n=0}^\infty$ such that every positive integer has a unique $f$-decomposition using $\{a_n\}$. Further, if $f$ is periodic, then the unique increasing sequence $\{a_n\}$ that corresponds to $f$ satisfies a linear recurrence relation. Previous research only handled recurrence relations with no negative coefficients. We find a function $f$ that yields a sequence that cannot be described by such a recurrence relation. Finally, for a class of functions $f$, we prove that the number of summands in the $f$-decomposition of integers between two consecutive terms of the sequence converges to a normal distribution.
\end{abstract}
\maketitle

\setcounter{equation}{0}

\tableofcontents


\section{Introduction}
The Fibonacci numbers are a very well known sequence, whose properties have fascinated mathematicians for centuries. Zeckendorf \cite{Ze} proved an elegant theorem stating that every positive integer can be written uniquely as the sum of non-consecutive Fibonacci numbers $\{F_n\}$, where\footnote{We don't start the Fibonacci numbers $F_1=1$, $F_2=1$, $F_3=2$ because doing so would lead to multiple decompositions for some positive integers.} $F_1=1$, $F_2=2$ and $F_{n} = F_{n-1} + F_{n-2}$. More is true, as the number of summands for integers in $[F_n, F_{n+1})$ converges to a normal distribution as $n\to\infty$. These results have been generalized to Positive Linear Recurrence Relations of the form
\begin{align} G_{n+1} \ = \ c_1 G_n + \cdots + c_L G_{n+1-L}, \end{align}
where $L, c_1, \dots, c_L$ are non-negative and $L, c_1$ and $c_L$ are positive. For every such recurrence relation there is a notion of ``legal decomposition'' with which all positive integers have a unique decomposition as a non-negative integer linear combination of terms from the sequence, and the distribution of the number of summands of integers in $[G_n, G_{n+1})$ converges to a Gaussian. There is an extensive literature for this subject; see \cite{Al,BCCSW,Day,GT,Ha,Ho,Ke,Len,MW1,MW2} for results on uniqueness of decomposition, \cite{DG,FGNPT,GTNP,KKMW,Lek,LamTh,MW1,St} for Gaussian behavior, and \cite{BBGILMT} for recent work on the distribution of gaps between summands.

An alternative definition of the Fibonacci sequence can be framed in terms of the Zeckendorf non-consecutive condition: The Fibonacci sequence (beginning $F_1=1$, $F_2=2$) is the unique increasing sequence of natural numbers such that every positive integer can be written uniquely as a sum of non-consecutive terms from the sequence. This is a special case of our results (described below). The condition that no two terms in the decomposition may be consecutive is the notion of legal decomposition in the case of Zeckendorf decompositions. In this paper, we encode notions of legal decomposition by a function $f:\N_0\to\N_0$.

\begin{defi} Given a function $f:\N_0\to\N_0$, a sum $x = \sum_{i=0}^k a_{n_i}$ of terms of $\{a_n\}$ is an \emph{$f$-decomp\-osition of $x$ using $\{a_n\}$} if for every $a_{n_i}$ in the $f$-decomposition, the previous $f(n_i)$ terms ($a_{n_i-f(n_i)}$, $a_{n_i-f(n_i)+1}$, $\dots$, $a_{n_i-1}$) are not in the $f$-decomposition. \end{defi}

We prove the following theorems about $f$-decompositions.

\begin{thm}\label{thm:mainfdecompexists} For any $f:\N_0\to\N_0$, there exists a sequence of natural numbers $\{a_n\}$ such that every positive integer has a unique legal $f$-decomposition in $\{a_n\}$.
\end{thm}

\begin{thm}\label{thm:mainfdcompunique} Let $f:\N_0\to\N_0$ be a function, and let $\{a_n\}$ be a sequence where $a_n = 1$ for $n<1$ and $a_n = a_{n-1}+a_{n-1-f(n-1)}$ for $n\ge 1$.  Then $\{a_n\}_{n=0}^\infty$ is the only increasing sequence of natural numbers in which every positive integer has a unique legal $f$-decomposition.
\end{thm}

As we study various such sequences, it is useful to give them a name.

\begin{defi}\label{def:fseq}
Given $f:\N_0\to\N_0$, the associated \emph{$f$-sequence} is the unique increasing sequence of natural numbers $\{a_n\}$ such that every positive integer has a unique $f$-decomposition.
\end{defi}

If we let $f$ be the constant function $f(n)=1$ for all $n\in\N_0$, we get the Zeckendorf condition\footnote{We say $f(0)=1$ only for notational convenience. Note that it is redundant as there are no terms before $a_0$ in the sequence.} that consecutive terms of the sequence may not be used. Hence the Fibonacci numbers are the $f$-sequence associated with the constant function $f(n)=1$ for all $n\in\N_0$.

The Fibonacci sequence is a solution to a recurrence relation. For certain $f$, we can prove similar connections between $f$-sequences and linear recurrence relations.

\begin{thm}\label{thm:mainfseqrecurrence} If $f(n)$ is periodic, then the associated $f$-sequence $\{a_n\}$ is described by a linear recurrence relation.\end{thm}

In Sections \ref{section:radixandfactorial} and \ref{section:3bingaussian} we consider various kinds of $f$-decompositions and study the distribution of the number of summands in the $f$-decomposition of integers picked in an interval. We find that these distributions converge in distribution\footnote{While we work with moment generating functions, since these functions converge well we could multiply the arguments by $i=\sqrt{-1}$ and obtain convergence results about the characteristic functions. By showing the moment generating functions converge pointwise, by the L\'evy continuity theorem we obtain convergence in distribution.} to a normal distribution for suitable growing intervals, which we now describe.

\begin{defi}  Let $f:\N_0\to\N_0$ be the function defined as \begin{equation} \{f(n)\} = \{\underbracket{0,}\underbracket{0,1,}\underbracket{0,1,2,}\underbracket{0,1,2,3,}\dots\},\end{equation} where each ``bin'' is one term wider than the previous, each bin begins with 0, and $f$ increases by exactly 1 within bins.
We say that the $f$-decomposition of $x \in \N$ for the function $f$ above is the \emph{Factorial Number System Representation} of the natural number $x$.
\end{defi}

\begin{thm}\label{thm:factorialgaussian}
Let the random variable $X_n$ denote the number of summands in the Factorial Number System Representation of an integer picked randomly from $\left[0, (n+1)!\right)$ with uniform probability. If we normalize $X_n$ as $X_n'$, so that $X_n'$ has mean 0 and variance 1, then $X_n'$ converge in distribution to the standard normal distribution as $n\to\infty$.
\end{thm}

The above theorem immediately implies the well-known result that the Stirling numbers of the first kind are asymptotically normally distributed (see Corollary \ref{cor:unsignedstirlinggauss}).

\begin{defi}
        Let $f:\N_0\to\N_0$ be a periodic function defined by\footnote{Again, $f(0)=1$ for convenience.} \begin{equation} \{f(n)\}\ =\  \{\underbracket{1,1,2,\dots,b-1,}\underbracket{1,1,2,\dots,b-1,}\dots\}.\end{equation} Let $\{a_n\}$ be the $f$-sequence that corresponds to this function $f$. We say that the $f$-decomposition of $x \in \N$ for the function $f$ above is the \emph{$b$-bin representation} of $x$.
\end{defi}

\begin{thm}\label{thm:gaussianbehaviorbbindecomp}
Let the random variable $X_n$ denote the number of summands in the $b$-bin representation of an integer picked at random from $[0, a_{bn})$ with uniform probability. Normalize $X_n$ as $Y_n = (X_n-\mu_n)/\sigma_n$, where $\mu_n$ and $\sigma_n$ are the mean and variance of $X_n$ respectively. If $b\geq 3$, $Y_n$ converges in distribution to the standard normal distribution as $n\to\infty$.
\end{thm}

\section{Constructing $f$-Sequences}

\subsection{Existence and Uniqueness Results}

Our proof of Theorem \ref{thm:mainfdecompexists} is constructive, and for any $f:\N_0\to\N_0$ gives a sequence $\{a_n\}$ such that every positive integer has a unique $f$-decomposition using $\{a_n\}$. This is an analogue to Zeckendorf's Theorem \cite{Ze}.

\begin{proof}[Proof of Theorem \ref{thm:mainfdecompexists}]
Let $a_0 = 1$. For $n\geq1$, define
\begin{align}a_n = a_{n-1} + a_{n-1-f(n-1)},\label{eqn:generatesequencefromf}\end{align} where $a_n$ may be assumed to be 1 when $n<0$. Notice that $\{a_n\}$ is a strictly increasing sequence. We use this definition of the sequence to show that all integers in $[a_m, a_{m+1})$ have an $f$-decomposition in $\{a_n\}_{n=0}^m$ for all $m \in \N$. We proceed by induction.

The sequence always begins $a_0 = 1, a_1=2$. Thus all integers in $[a_0, a_1)$ can be legally decomposed in $\{a_n\}$. For $m>0$, recall that $a_{m+1} = a_{m} + a_{m-f(m)}$.\\ \

\textbf{Case I:} If $m-f(m)<0$, we have $a_{m+1} = a_m + 1$. Therefore $[a_m, a_{m+1}) = \{a_m\}$ and $a_m$ has an $f$-decomposition.\\ \

\textbf{Case II:} If $m-f(m) \geq 0$, consider any $x \in [a_m, a_{m+1}) = [a_m, a_m + a_{m-f(m)})$. Therefore $x - a_m \in [0, a_{m-f(m)})$. By the induction hypothesis, $x-a_m$ has an $f$-decomposition in $\{a_n\}_{n=0}^{m-f(m)-1}$. Since no terms from $\{a_n\}_{n=m-f(m)}^{m-1}$ are in the $f$-decomposition of $x-a_m$, we may use this $f$-decomposition and add $a_m$ to get $x$. \\ \

We now prove uniqueness of $f$-decompositions. Let $S = \sum_{i=1}^k a_{n_i}$ be an $f$-decomposition with $n_1 > n_2 > \dots > n_k$. We first show that $S \in [a_{n_1}, a_{n_1 +1})$. It is clear that $S \geq a_{n_1}$. We show by induction on $n_1$ that $S < a_{n_1+1}$. If $n_1 = 0$, then $S = a_0 = 1$. If $n_1 \geq 1$, then $S = a_{n_1} + \sum_{i = 2}^k a_{n_i}$. By the induction hypothesis, $\sum_{i=2}^k a_{n_i} < a_{n_2 + 1}$. We know from our notion of $f$-decomposition that $n_2 \leq n_1 - f(n_1)- 1$. Since $\{a_n\}$ is increasing, $\sum_{i=2}^k a_{n_i} < a_{n_1-f(n_1)}$. This gives us $S = \sum_{i=1}^k a_{n_i} < a_{n_1} + a_{n_1 - f(n_1)} = a_{n_1 + 1}$.

Consider two $f$-decompositions $x=\sum_{i=1}^k a_{n_i} = \sum_{j=1}^l a_{m_j}$ for the same positive integer $x$ with $n_1 > n_2 > \dots > n_k$ and $m_1 > m_2 > \dots > m_l$. Assume for the sake of contradiction that $\{n_1, n_2, \dots, n_k\} \neq \{m_1, m_2, \dots, m_l\}$. Let $h$ be the smallest natural number such that $n_h \neq m_h$ (it is clear that such an $h$ exists with $h\leq k$ and $h \leq l$). We have $\sum_{i=1}^{h-1} a_{n_i} = \sum_{j=1}^{h-1} a_{m_j}$. Thus $\sum_{i=h}^k a_{n_i} = \sum_{j=h}^l a_{m_j}$. However, as shown above, $\sum_{i=h}^k a_{n_i} \in [a_{n_h}, a_{n_h + 1})$ and $\sum_{j=h}^l a_{m_j} \in [a_{m_h}, a_{m_h + 1})$, which are disjoint intervals. Therefore $\sum_{i=1}^k a_{n_i} \neq \sum_{j=1}^l a_{m_j}$, a contradiction.
\end{proof}

Theorem \ref{thm:mainfdecompexists} gives us a construction for $\{a_n\}$. Theorem \ref{thm:mainfdcompunique} tells us that $\{a_n\}$ is the only increasing sequence of natural numbers for a given $f(n)$. As mentioned in Definition \ref{def:fseq}, we call this sequence the \emph{$f$-sequence}.

\begin{proof}[Proof of Theorem \ref{thm:mainfdcompunique}]
We proceed by induction. Let $\{a'_n\}$ be an increasing sequence such that every positive has a unique $f$-decomposition using $\{a_n\}$. Since 1 cannot be written as a sum of other positive integers, we require $a'_0 = a_0 = 1$.  Now suppose that $a'_i = a_i$ for each $0\le i\le m-1$.  As shown in the proof of Theorem \ref{thm:mainfdecompexists}, each $x\in[0,a_m)$ has a unique $f$-decomposition in $\{a'_n\}_{n=0}^{m-1}$.  Thus, $a'_m \ge a_m$; otherwise, it would not have a unique $f$-decomposition.  On the other hand, if $a'_m > a_m$, then the integer with value $a_m$ does not have an $f$-decomposition.  Thus, $a'_m = a_m$.
\end{proof}

\subsection{Linear Recurrences and Periodic $f$}

Now that we have the unique increasing sequence for a given $f:\N_0\to\N_0$, we prove Theorem \ref{thm:mainfseqrecurrence}, which says that this sequence satisfies a linear recurrence relation if $f$ is periodic. The following lemma is a key ingredient in the proof; see for example \cite{LaTa} for a proof.

\begin{lem}
Let $a_{1,n}$, $a_{2, n}$, \dots, $a_{b, n}$ be linearly recurrent sequences, whose recurrence relations need not be equal. Then the sequence $\{a_n\}_{n=0}^\infty$ constructed by interlacing the sequences as $a_{1,1}$, $a_{2,1}$, \dots, $a_{b, 1}$, $a_{1, 2}$, $a_{2, 2}$, \dots is also linearly recurrent. \label{lem:interlacedsequences}
\end{lem}

To see why this is true, note that if $f_i(x)$ is the characteristic polynomial of the sequence $\{a_{i,n}\}_{n=0}^\infty$, then each subsequence $\{a_{i,n}\}$ satisfies the recurrence relation whose characteristic polynomial is $\prod_{i=1}^{b} f_i(x)$ and the interlaced recurrence relation satisfies the recurrence relation whose characteristic polynomial is $\prod_{i=1}^{b} f_i(x^b)$.

\begin{proof}[Proof of Theorem \ref{thm:mainfseqrecurrence}]
Let $p$ be the period of $f$ and let $b$ the smallest integer multiple of $p$ such that $b \geq f(n)+1$ for all $n\in\N_0$. When $n \equiv r \bmod b$ for some $r \in \N_0$, we have $f(n) = f(r)$ and we have a linear expression  (i.e., $a_{n+1} - a_{n} - a_{n-f(r)} = 0$) for $a_{n+1}$ in terms of the previous $b$ terms. We can write this as a vector of sufficient dimension ($b^2 + 1$ suffices):
\begin{equation}  \vec v_0 \ = \  [1 ~ -1 ~ \underbrace{0 ~ 0 ~ 0 ~ \dotsb ~ 0}_\text{$f(r) -1$ times} ~ -1 ~ 0 ~ \dotsb ]. \end{equation}

Such a recurrence relation exists for any $a_{n-i}$ where $i \in \N_0$ and $n\equiv r \bmod b$. When $i\leq b^2-b$ (i.e., when there is room in the vector to fit the recurrence relation), we can write the corresponding vector as
\begin{equation}  \vec v_i \ = \  [\underbrace{0 ~ 0 ~ \dotsb ~ 0}_\text{$i$ times} ~ 1 ~ -1 ~ \underbrace{0 ~ 0 ~ 0 ~ 0 ~ \dotsb ~ 0}_\text{$f(n-i) -1$ times} ~ -1 ~ 0 ~ \dotsb ]. \end{equation}

Our goal is to find a recurrence relation that holds for all $n \equiv r \bmod b$ and only uses terms whose indices are from the same residue class modulo $b$. We begin by finding a recurrence relation for each residue class, not necessarily the same recurrence relation. More precisely, when $n\equiv r\bmod b$, we claim that $a_n$ satisfies a recurrence relation of the form
\begin{align}
a_n \ = \  \sum_{i=1}^{b+1} c_i a_{n-bi}.\label{eqn:recrelsinslots}
\end{align}

Our proof of the above claim is algorithmic. In the proof, we index the coordinates of vectors starting at 0. We define \emph{bad coordinates} to be non-zero coordinates whose indices are not multiples of $b$. A vector without bad coordinates corresponds to a recurrence relation of the form \eqref{eqn:recrelsinslots}.

Let $\vec w_0 = \vec v_0$. Let $\vec u_0$ be a truncated copy of $\vec w_0$ containing only the coordinates between coordinate 0 and coordinate $b$, both exclusive. Hence $\vec u_0$ is of dimension $b-1$. We iteratively find $\vec w_1, \vec w_2, \dots, \vec w_{b-1}$ and $\vec u_1, \vec u_2, \dots, \vec u_{b-1}$ through the following algorithm, looping from $i=1$ to $i=b-1$.

\begin{itemize}
\item All bad coordinates have index between $(i-1)b$ and $ib$. Use $\vec v_{(i-1)b+1}$, $\vec v_{(i-1)b+2}$, \dots, $\vec v_{ib-1}$ to cancel all coordinates with index from $(i-1)b +1$ to $ib -1$ so that they are all zero. This yields the vector $\vec w_i$. All bad coordinates of $\vec w_i$ have index between $ib$ and $(i+1)b$ because the degree of the recurrence relation corresponding to vectors $\vec v_0$, $\vec v_1$, \ldots, $\vec v_{b^2-b}$ are all at most $b$.
\item If $\vec w_i$ has zeros in all coordinates with index from $ib+1$ and $(i+1)b-1$, we have a recurrence relation of the form \eqref{eqn:recrelsinslots}, where the only terms with non-zero coefficients are of the form $a_{n-ib}$ for $i\in\N_0$. We are done. Otherwise, we continue.
\item Let $u_i$ be a truncation of $\vec w_i$, containing only coordinates with indices from $ib+1$ to $(i+1)b-1$ (i.e., the bad coordinates), so $\vec w_i$ is a vector of dimension $(b-1)$.
\end{itemize}

Let $U = \{\vec u_0, \vec u_1, \dots, \vec u_{b-1} \}$. Notice $U$ has $b$ vectors, each of dimension $b-1$. Since $U$ cannot be a linearly independent set, we have a non-trivial solution to $\lambda_0 \vec u_0 +\lambda_1 \vec u_1 + \dots + \lambda_{b-1} \vec u_{b-1} = \vec 0$. We want to find a non-trivial linear combination of shifted versions of the vectors $\vec w_0, \vec w_1, \dots, \vec w_{b-1}$ so that it contains no bad coordinates. Let $T_b:\R^{b^2+1}\to\R^{b^2+1}$ shift the coordinates of a vector to the right by $b$ coordinates. Since each coordinate remains in the same residue class modulo $b$, the relation corresponding to a vector with shifted coordinates still holds. All the bad coordinates in $T_b^{b-i-1}(\vec w_i)$ have index between $b^2 -b +1$ and $b^2 -1$. Thus $\sum_{i=0}^b \lambda_i T_b^{b-1-i}(\vec w_i)$ is of the form \eqref{eqn:recrelsinslots} because the bad coordinates cancel to give 0.

Notice that this sum is not the zero vector because the first coordinate of $\vec w_i$ is always 1 and for the largest $i$ such that $\lambda_i \neq 0$, the first non-zero coordinate is further left than the first non-zero coordinate of any other vector in the sum (because the other vectors are shifted even further to the right).

For each residue class modulo $b$, we have a recurrence relation that describes the subsequence of $\{a_n\}$ with indices from that residue class. Thus we can write $\{a_n\}$ as $b$ linearly recurrent sequences interlaced together. By Lemma \ref{lem:interlacedsequences}, $\{a_n\}$ is also a linearly recurrent sequence. \end{proof}

As the above proof is algorithmic, a detailed step-by-step example is provided in Appendix \ref{appendix:exampleoflinearrec} for illustrative purposes.

\section{Radix Representation and the Factorial Number System}
\label{section:radixandfactorial}

Numerous radix representations can be interpreted as $f$-decompositions. The most basic are base $b$ representations, which can be interpreted as $f$-decompositions where $f(n) = n \bmod (b-1)$.

\begin{exa} Consider the following function $f$:
\begin{equation}  \{ f(n) \} \ = \  \{ \underbracket{0,1,2,3,}\underbracket{0,1,2,3,}\underbracket{0,1,2,3,}\dots \}.\end{equation}
The associated $f$-sequence is
\begin{equation}  \{a_n\} \ = \  \{ \underbracket{1,2,3,4,}\underbracket{5,10,15,20,}\underbracket{25, 50, 75, 100,}\dots \}. \end{equation}
An $f$-decomposition permits at most one summand from each ``bin''. In the base 5 representation of any natural number, the $i^\text{th}$ digit denotes one of $1\cdot 5^{i-1}$, $2\cdot 5^{i-1}$, $3\cdot 5^{i-1}$, $4\cdot 5^{i-1}$, or 0. This highlights the relationship between these $f$-decompositions and base 5 representations.
\end{exa}

The Factorial Number System is a mixed radix numeral system, where the $i^\text{th}$ radix corresponds to a place value of $(i-1)!$ and the digits go from 0 to $i-1$. Like base $b$ decompositions, this can be interpreted as $f$-decompositions, where
\begin{equation}  \{a_n\} \ = \  \underbracket{1,}\underbracket{2,4,}\underbracket{6,12,18,}\underbracket{24,48,72,96,}\underbracket{120,240,360,480,600} \end{equation}
and
\begin{equation}  \{f(n)\} \ = \  \{\underbracket{0,}\underbracket{0,1,}\underbracket{0,1,2,}\underbracket{0,1,2,3,}\underbracket{0,1,2,3,4,}\dots \}.\end{equation}
Here, too, at most one term from each bin may be chosen in an $f$-decomposition.

Let the random variable $X_n$ denote the number of summands in the Factorial Number System Representation of an integer picked randomly from $\left[0, (n+1)!\right)$ with uniform probability. We now prove the distribution of $X_n$ converges in distribution to a normal distribution.

\begin{proof}[Proof of Theorem \ref{thm:factorialgaussian}]
We can write $X_n$ as a sum of random variables $Y_1, Y_2, \dots, Y_n$, where $Y_i$ represents the number of summands from the $i^\text{th}$ bin. For any integer $i \in [1, n]$, there are an equal number of integers in the interval $\left[n!, (n+1)!\right)$ with a given term from the $i^\text{th}$ bin in the integer's decomposition as there are integers with no terms from the $i^\text{th}$ bin in its decomposition. Therefore each digit from 0 to $i$ has equal probability in the $i^\text{th}$ bin, independent of the terms in other bins.

This proves that each $Y_i$ is independent from the other $Y_j$ random variables, and $P(Y_i = 1)$ $=$ $\frac{i-1}{i}$, and $P(Y_i = 0)$ $=$ $\frac{1}{i}$. We apply the Lyapunov Central Limit Theorem (see for example \cite[pp.\ 371]{Billingsley}) to show that $X_n$ converges in distribution to a Gaussian distribution as $n\to\infty$; note we cannot apply the standard version as our random variables are not identically distributed.

Let $s_n^2 = \sum_{i=1}^n \sigma_i^2$. The Lyapunov Central Limit Theorem states that if there exists some $\delta>0$ such that
\begin{align}
\lim_{n\to\infty}\frac{1}{s_n^{2+\delta}}\sum_{i=1}^n \E \left ( | Y_i - \mu_i|^{2+\delta} \right ) \ = \  0,
\end{align}
then the Central Limit Theorem for $Y_i$ holds. That is, $\frac 1N \sum_{i=1}^N Y_i$ converges in distribution to a normal distribution.

We let $\mu_i = \E (Y_i)$ and $\sigma_i^2 = \E(Y_i^2-\mu_i^2)$. From our definition of $Y_i$ above, we can see
\begin{align}
\mu_i \ = \  & \E(Y_i) \ = \  \frac{i-1}i\nonumber\\
\sigma_i^2 \ = \  & \E(Y_i^2-\mu_i^2) \ = \  \frac{i-1}i - \frac{(i-1)^2}{i^2} \ = \  \frac{i-1}{i^2}.
\end{align}
We show we may take $\delta=2$. 
\begin{align}
    \lim_{n\to\infty} \frac{1}{s_n^{2+\delta}}\sum_{i=1}^n \E \left ( | Y_i - \mu_i|^{2+\delta} \right ) \ = \  & \lim_{n\to\infty} \frac{1}{\left( \sum_{i=1}^n \frac{i-1}{i^2} \right)^2}\sum_{i=1}^n \left ( \frac{i-1}{i} \left( \frac1i \right)^4 + \frac1i \left(\frac{i-1}{i}\right)^4 \right )\nonumber\\
\ = \  & \lim_{n\to\infty} \frac{1}{\left( \sum_{i=1}^n \frac{i-1}{i^2} \right)^2}\sum_{i=1}^n \frac{i^4 - 4i^3+ 6i^2 - 3i}{i^5}\nonumber\\
\ = \  & \lim_{n\to\infty} \frac{O(\log n)}{\log^2(n) + o(\log^2 n)} \ = \  0.
\end{align}
Thus the Lyapunov Central Limit Theorem conditions are met, and $X_n'$, the normalization of $X_n$, converges in distribution to the standard normal distribution as $n\to\infty$. \end{proof}

Similar results for base $b$ representations follow trivially from the classical Central Limit Theorem. As a corollary to Theorem \ref{thm:factorialgaussian}, we can show the following previously known result (see \cite{FS}).

\begin{cor}\label{cor:unsignedstirlinggauss} The unsigned Stirling numbers of the first kind $\stir{n}{k}$ are asymptotically normally distributed.\end{cor}

\begin{proof}
Let $p_{n,k}$ be the number of integers in $[0, (n+1)!)$ whose $f$-decomposition contains exactly $k$ summands. It is clear that these summands all come from the first $n$ bins. To count the number of ways to choose $k$ summands from $n$ bins, we select $k$ of the $n$ bins. For each bin chosen (say we choose the $i^\text{th}$ bin), we can choose any of the $i$ elements in the bin.

Let $\mathcal{I} = \{1,2,\dots,n\}$. We can write $p_{n,k}$ as
\begin{align}
p_{n,k} \ = \  \sum_{S \subset \mathcal{I}, \  |S|= k} \ \  \prod_{s\in S} s.
\end{align}
Notice that defined in this way, $p_{n,k}$ is the coefficient of $x^{n-k+1}$ in the expansion of $x(x+1)$ $(x+2)$ $\cdots$ $(x+n)$. This allows us to write $p_{n,k}$ as the Stirling number $\stir{n+1}{n-k+1}$. Therefore, by Theorem \ref{thm:factorialgaussian} the unsigned Stirling numbers of the first kind are asymptotically normally distributed.
\end{proof}

\section{$b$-Bin Decompositions}\label{section:3bingaussian}

\subsection{Zeckendorf's Theorem for $b$-Bin Decompositions}\label{sec:zeckbbindecomp}

Previous work on generalizing Zeckendorf's Theorem (see for example \cite{MW1,MW2,St}) only handled linear recurrence relations with non-negative coefficients. Not only do some types of $f$-decompositions have notions of legal decomposition that do not result from any Positive Linear Recurrence Relations \cite{MW1,MW2} (called $G$-ary digital expansions in \cite{St}), but we can also find functions $f$ that correspond to recurrence relations with some negative coefficients, which is beyond the scope of previous work. In this section we explore a class of $f$ which include some of these new cases.

Let $b\geq3$ be an integer. We partition non-negative integers as ``bins'' of $b$ consecutive integers (i.e., $\{0,1,2,\dots,b-1\}$, $\{b, b+1, \dots, 2b-1\}, \dots$).

\begin{defi}[$b$-bin Decomposition] Let $b \ge 3$ be an integer. A \emph{$b$-bin decomposition} of a positive integer is legal if the following two conditions hold.
\begin{enumerate}
\item No two distinct terms $a_i, a_j$ in a decomposition can have indices $i, j$ from the same bin.
\item No consecutive terms from the sequence may be in the decomposition.
\end{enumerate}
\end{defi}

We can interpret this as an $f$-decomposition, where
\begin{align}
f(n) \ = \  \max \{ 1, n \bmod b \}.
\end{align}

\begin{exa}
When $b=3$, the resulting sequence is
\begin{equation}  \{a_n \} \ = \   \underbracket{1,2,3}, \underbracket{4, 7, 11}, \underbracket{15, 26, 41}, \underbracket{56, 97, 153}, \underbracket{209, 362, 571}, \dots. \end{equation}
Notice the similarity to base-$b$ representation discussed in Section \ref{section:radixandfactorial}.
\end{exa}

We begin by finding a recurrence relation for the sequence $\{a_n\}_{n=0}^\infty$ resulting from constant-width bins.

\begin{prop}
If $f(n) = \max \{ 1, n \bmod b \}$ (and thus we have a $b$-bin decomposition), then the associated $f$-sequence $\{a_n\}_{n=0}^\infty$ satisfies $a_n = (b+1)a_{n-b} - a_{n-2b}$.
\end{prop}
\begin{proof}
The proof follows by induction. We prove two base cases and show that the recurrence relation holds for the remaining terms by applying the relation \eqref{eqn:generatesequencefromf} that was used to generate the sequence, regrouping terms, and finding the desired recurrence.

We begin by finding the initial $2b+1$ terms of the sequence.
\begin{itemize}
\item For $n\leq b$, we have $a_n = n + 1$.
\item For $n = b+1$, we have $b-f(b) = b - 1$. Thus $a_{n} = a_{b} + a_{b-1} = 2b+1$.
\item For $b+2 \leq n \leq 2b$, we have $n-1-f(n-1) = b$. Therefore $a_n = a_{n-1} + a_b$, which gives us $a_n = 2b+1+ (n-b-1)(b+1)$.
\item For $n=2b+1$, we have $n-1-f(n-1) = 2b - 1$. Therefore, $a_n = a_{2b} + a_{2b-1} = 2b+1 + (b-1)(b+1) + 2b+1 + (b-2)(b+1) = 2b^2 + 3b - 1$.
\end{itemize}

First, we prove the base cases and show that the recurrence relation holds for $a_{2b}$ and $a_{2b+1}$. We have
\begin{align} a_{2b} \  = \  b^2+2b\ =\ (b+1)(b+1) - 1 \ = \  (b+1)a_{2b-b} - a_{2b-2b}\end{align} and \begin{align} a_{2b+1} \  = \  2b^2 + 3b - 1 \ = \  (b+1)(2b+1) - 2 \ = \  (b+1)a_{(2b+1)-b} - a_{(2b+1)-2b}.\end{align}

For any $n'>2b+1$, our induction hypothesis states that $a_n = (b+1)a_{n-b} - a_{n-2b}$ for all $n \in [2b, n')$. We use $n'-1-f(n'-1) \geq 2b$ for $n'>2b+1$.

We constructed the sequence as $a_{n'} = a_{n'-1} + a_{n'-1-f(n'-1)}$. Therefore
\begin{align}
a_{n'} \  = \  & ( (b+1) a_{n'-1-b}-a_{n'-1-2b}) + ( (b+1) a_{n'-1-f(n'-1)-b} - a_{n'-1-f(n'-1)-2b})\nonumber\\
   = \  & (b+1)(a_{n'-1-b} + a_{n'-1-f(n'-1)-b}) - (a_{n'-1-2b} + a_{n'-1-f(n'-1)-2b}).
\end{align}

As $f$ is periodic, $f(n'-1) = f(n'-b-1)$. Thus
\begin{align}
a_{n'} \  = \  & (b+1)(a_{n'-b-1} + a_{n'-b-1-f(n'-b-1)}) - (a_{n'-2b-1} + a_{n'-2b-1-f(n'-2b-1)})\nonumber\\
= \  & (b+1)a_{n'-b} - a_{n'-2b}.
\end{align}

\end{proof}

When $b=3$, the corresponding recurrence relation is $a_n = 4a_{n-3} - a_{n-6}$. In Appendix \ref{appendix:3binnegativecoeffients} we prove that there is no linear recurrence relation with non-negative coefficients that this sequence satisfies. Thus this $f$-sequence is a new sequence whose behavior cannot be analyzed by previous methods.

While previous techniques cannot handle this sequence, it is still natural to ask if we obtain Gaussian behavior, as this sequence has similar properties to previously studied sequences in terms of uniqueness of decomposition. The answer is yes, and we prove below that the number of summands for integers chosen from $[0, a_{bn})$ converges in distribution to being normally distributed as $n\to\infty$. In \S\ref{section:genfunc} we calculate the generating function for the number of summands, then compute the mean and variance in \S\ref{sec:computingmeanvar}, and finally prove Gaussianity in \S\ref{sec:gaussianbehaviorbbindecomp}.

\subsection{Generating Function}\label{section:genfunc}

We take a combinatorial approach to finding the distribution of the number of summands for integers in $[0, a_{bn})$. We begin by finding a two-dimensional sequence for the number of integers that can be written as the sum of $k$ summands chosen from the first $n$ bins.

\begin{prop}Let $p_{n,k}$ be the number of integers that are the sum of $k$ summands from $n$ consecutive bins. We have
\begin{align}
p_{n,k} &\ = \  p_{n-1, k} + bp_{n-1, k-1} - p_{n-2, k-2}.\label{eqn:RecRelP}
\end{align}
\end{prop}

\begin{proof}
We count all possible ways to legally choose $k$ summands from the first $n$ bins. We call the bin containing $\{a_{bn-(b-1)}, a_{bn-(b-2)}, \dots, a_{bn}\}$ the first bin, the bin containing $\{a_{b(n-1)-(b-1)}, \dots, \allowbreak a_{b(n-1)}\}$ the second bin, and so on. We derive a recurrence relation for the $\{p_{n,k}\}$ by counting how many valid choices there are.

In the arguments below we assume $n \ge 2$ so that there are at least two bins. There are two ways to have a contribution to $p_{n,k}$.

\begin{enumerate}
\item We may choose no summands from the first bin. Therefore we need $k$ summands from $n-1$ bins and there are exactly $p_{n-1,k}$ ways of doing so.

\item We may choose any of the $b$ summands in the first bin, leaving $k-1$ terms to choose from $n-1$ bins. There are $bp_{n-1,k-1}$ ways of doing that. However, this argument counts some illegal decompositions. We are not allowed to choose the last element of the first bin and the first element of the second bin. There are exactly $p_{n-2,k-2}$ such decompositions because after choosing these two terms, $k-2$ terms remain to be chosen from the $n-2$ remaining bins, and there are no restrictions on our choice from the remaining bins.
\end{enumerate}

Hence for $n \ge 2$ we have the following recurrence relation (in two variables):
\begin{align}
p_{n,k} &\ = \  p_{n-1, k} + bp_{n-1, k-1} - p_{n-2, k-2}.
\end{align}
\end{proof}

This recurrence relation allows us to compute a closed form expression for $F(x,y)$, the generating function of the $p_{n,k}$'s.

\begin{prop}Let $F(x,y) = \sum_{n,k\geq0} p_{n,k} x^ny^k$. The closed form expression of $F(x,y)$ is
\begin{align}
F(x,y) & \ = \  \frac{1}{1 - x - bxy + x^2 y^2}\label{eqn:GeneratingFunctionF}
\end{align}
\end{prop}

\begin{proof}
We have $p_{0,0}=1$, $p_{0,k} = 0$ if $k>0$, and $p_{n,k} = 0$ if $n < 0$. Using the recurrence relation \eqref{eqn:RecRelP}, after some algebra we find
\begin{equation}F(x,y) \ = \  xF(x,y) + bxyF(x,y) - x^2y^2 F(x,y) + 1,\end{equation}
which yields (\ref{eqn:GeneratingFunctionF}).
\end{proof}
We now find the coefficient of $x^n$ in $F(x,y)$, which we denote by $g_n(y)$.
\begin{prop}
Let $g_n(y) =  \sum_{k\geq 0} p_{n,k} y^k$, which is the coefficient of $x^n$ in the generating function of the $p_{n,k}$'s. For $b>2$, the closed form expression of $g_n(y)$ is
\begin{align}
   g_n(y) \ = \  \frac{\left (by+1 + \sqrt{(b^2-4)y^2 + 2by+1}\right ) ^{n+1} - \left (by+1 - \sqrt{(b^2-4)y^2 + 2by+1}\right )^{n+1}}{2^{n+1}\sqrt{(b^2-4)y^2 + 2by+1}}.
\end{align}
\end{prop}

\begin{proof}
Let $x_1(y), x_2(y)$ be the solutions for $x$ of $1 - x - bxy + x^2 y^2=0$. It is clear that
\begin{align}
y^2 (x-x_1(y))(x-x_2(y)) \ = \  1 - x - bxy + x^2 y^2.\label{eqn:bbinpolyfromroots}
\end{align}
It is easy to verify that for all $b>2$ and non-negative $y$, $x_1(y)$ and $x_2(y)$ are unequal. Thus
\begin{align}
F(x,y) \ = \  \frac{1}{y^2(x_1(y) - x_2(y))} \left [ \frac{1}{x-x_1(y)} - \frac{1}{x-x_2(y)} \label{eqn:FAsPartialFractions} \right].
\end{align}
We use the geometric series formula to expand \eqref{eqn:FAsPartialFractions}, obtaining
\begin{align}
F(x,y) & \ = \  \frac{1}{y^2(x_2(y) - x_1(y))} \left [ \frac{1/x_1(y)}{1-x/x_1(y)} - \frac{1/x_2(y)}{1-x/x_2(y)} \right]\nonumber\\
   & \ = \  \frac{1}{y^2(x_2(y) - x_1(y))} \sum_{i=0}^\infty \left [ \frac{1}{x_1} \left (\frac{x}{x_1}\right)^i - \frac{1}{x_2} \left ( \frac{x}{x_2}\right)^i \right ].\label{eqn:FAsPowerSeries}
\end{align}

Using the quadratic formula, we find the roots of the quadratic equation in \eqref{eqn:bbinpolyfromroots} are
\begin{align}
x_1(y) & \ = \  \frac{by+1 - \sqrt{(b^2-4)y^2 + 2by + 1}}{2y^2}\nonumber\\
x_2(y) & \ = \  \frac{by+1 + \sqrt{(b^2-4)y^2 + 2by + 1}}{2y^2}\label{eqn:bbinrootsofpoly}.
\end{align}
If we write $F(x,y)$ as a power series of $x$ and define $g_n(y)$ to be the coefficient of $x^n$ in $F(x,y)$, then \eqref{eqn:FAsPowerSeries} implies that
\begin{equation}
g_n(y) \ = \  \frac{1}{y^2(x_2(y)-x_1(y))} \left (\frac{x_2(y)^{n+1} - x_1(y)^{n+1}}{(x_1(y)x_2(y))^{n+1}} \right ).
\end{equation}
Using \eqref{eqn:bbinrootsofpoly} we find
\begin{align}
g_n(y) & \ = \  \frac{y^{2n+2}(x_2(y)^{n+1} - x_1(y)^{n+1})}{\sqrt{5y^2+6y+1}}\nonumber\\
   & \ = \  \frac{\left (by+1 + \sqrt{(b^2-4)y^2 + 2by+1}\right ) ^{n+1} - \left (by+1 - \sqrt{(b^2-4)y^2 + 2by+1}\right )^{n+1}}{2^{n+1}\sqrt{(b^2-4)y^2 + 2by+1}},
\end{align} which completes the proof. \end{proof}

\subsection{Computing The Mean and Variance}\label{sec:computingmeanvar}

Let $X_n$ be the random variable denoting the number of summands in the unique $b$-bin decomposition of an integer chosen uniformly from $[0, a_{bn})$. The integers in $[0, a_{bn})$ are exactly those integers whose unique decomposition contains only terms from the first $n$ bins.

\begin{prop}The mean number of summands of $b$-bin decompositions, $\mu_n$, for integers in $[0, a_{bn})$ is
\begin{align}
\mu_n \ = \  \frac{\left(b^2+b-4+b \sqrt{b^2+2b-3}\right) n}{\sqrt{b^2+2b-3} \left(1+b+\sqrt{b^2+2b-3}\right)} + O(1).\label{eqn:ValueOfMeanApprox}
\end{align}\label{prop:ValueOfMean}
\end{prop}

\begin{proof}
The mean value $\mu_n$ of $X_n$ is
\begin{align}
\mu_n \ = \  \sum_{i=0}^n i P(X_n=i) \ = \  \sum_{i=0}^n i\frac{p_{n,i}}{\sum_{k=0}^n p_{n,k}} \ = \  \frac{g_n'(1)}{g_n(1)}.
\end{align}
Computing the value of $g_n'(1)/g_n(1)$, we get
\begin{align}
\frac{g_n'(1)}{g_n(1)} \: = \  & \frac{4 \sqrt{b^2+2 b-3}+(b^2+2b-3) (b^2+b-4+b \sqrt{b^2+2b-3}) n}{(b^2+2b-3)^{3/2} (1+b+\sqrt{b^2+2b-3})},
\end{align}
which gives us \eqref{eqn:ValueOfMeanApprox}.
\end{proof}

\begin{prop}The variance $\sigma_n^2$ of $X_n$ is
\begin{align}
\sigma_n^2 \ := \  \frac{(b^2+ b - 4) n}{(b^2+ 2 b-3 )^{3/2}} + O(1).\label{eqn:ValueOfVarianceApprox}
\end{align}\label{prop:ValueOfVariance}
\end{prop}

\begin{proof}
Similar to the computation of the mean, the variance $\sigma_n$ of $X_n$ can be computed as
\begin{align}
\sigma_n^2 \ = \  \sum_{i=0}^n (i-\mu_n)^2 P(X_n=i) \ = \  \sum_{i=0}^n i^2\frac{p_{n,i}}{\sum_{k=0}^n p_{n,k}} - \mu_n^2 \ = \
\frac{\left . \frac{d}{dy}\left[ y g_n'(y) \right] \right |_{y=1}}{g(1)} - \mu^2\label{eqn:ExprnForVariance}.
\end{align}
Computing \eqref{eqn:ExprnForVariance}, we get
\begin{align}
\sigma_n^2 \ = \  \: &\Big(2 (b^6 n+b^5 (5+\sqrt{-3+2 b+b^2}) n+b^4 (-2+4 \sqrt{-3+2 b+b^2} n)\nonumber\\
& +4 (-9+\sqrt{-3+2 b+b^2}+3 (-1+\sqrt{-3+2 b+b^2}) n)-2 b^3 (1+\sqrt{-3+2 b+b^2}\nonumber\\
&+(12+\sqrt{-3+2 b+b^2}) n)+b (2 (9+7 \sqrt{-3+2 b+b^2})+(35+\sqrt{-3+2 b+b^2}) n)\nonumber\\
&+b^2 (22-(5+16 \sqrt{-3+2 b+b^2}) n))\Big)/\Big((-3+2 b+b^2)^{5/2} (1+b+\sqrt{-3+2 b+b^2})^2\Big),
\end{align}
which gives us \eqref{eqn:ValueOfVarianceApprox}.
\end{proof}


\subsection{Gaussian Behavior}\label{sec:gaussianbehaviorbbindecomp}

In Section \ref{section:genfunc} we found a closed form expression for the generating function $g_n(y)$ of the sequence $p_{n,k}$ for any $n$. Our expansion of the generating function $g_n(y)$ allows us to explicitly find the moment generating function of $X_n$, which converges in distribution to a Gaussian and thus proves Theorem \ref{thm:gaussianbehaviorbbindecomp} (that the distribution of the number of summands converges to a Gaussian). Before we prove Theorem \ref{thm:gaussianbehaviorbbindecomp}, we need a lemma.

\begin{lem}The moment generating function $M_{Y_n}(t)$ of $Y_n$ is
\begin{align}
M_{Y_n}(t) \ = \  \mathbb{E}(e^{tY_n}) \ = \  \: \frac{g_n(e^{t/\sigma_n})e^{-t\mu_n/\sigma_n}}{g_n(1)}.
\end{align}
\label{lem:momgenfun}
\end{lem}

\begin{proof}
Our goal is to study the distribution of $X_n$ as $n\to \infty$, where \begin{align}P(X_n=k) \ = \   \frac{p_{n,k}}{\sum_{k\geq0} p_{n,k}}.\end{align}
Observe that
\begin{align}
\frac{g_n(e^t)}{g_n(1)} \ = \  \sum_{k\geq0} \frac{p_{n,k}e^{tk}}{\sum_{k\geq0} p_{n,k}} \ = \  \mathbb{E}(e^{tX}).
\end{align}
If we let $\mu_n$ and $\sigma_n$ be the mean and standard deviation of $X_n$ respectively, then we can normalize the generating function by letting $Y_n = \frac{X_n-\mu_n}{\sigma_n}$, which yields
\begin{align}
    M_{Y_n}(t) \ = \  \mathbb{E}(e^{tY_n}) \ = \  \sum_{k\geq0} \frac{p_{n,k}e^{t\frac{(k-\mu_n)}{\sigma_n}}}{\sum_{k\geq0}p_{n,k}} \ = \  \frac{g_n(e^{t/\sigma_n})e^{-t\mu_n/\sigma_n}}{g_n(1)}.
\label{eqn:generalmomentgenfun}
\end{align}
\end{proof}

We now prove Theorem \ref{thm:gaussianbehaviorbbindecomp}, which says that $Y_n$ converges in distribution to a normal distribution, by showing that the moment generating function $M_{Y_n}(t)$ of $Y_n$ converges pointwise to the standard normal as $n\to\infty$ (see for example \cite{Wi}).

\begin{proof}[Proof of Theorem \ref{thm:gaussianbehaviorbbindecomp}]
For convenience we set $r := t/\sigma_n$. Since $\sigma_n = c\sqrt{n} + o(\sqrt{n})$, where $c$ is some positive constant, we know that $r \to 0$ as $n\to \infty$ for a fixed value of $t$. Thus we may expand functions of $r$, such as $e^r$, using their power series representations. We begin by manipulating $g_n(e^r)$:
\begin{align}
g_n(e^r) \: = \  & \frac{(be^r + 1 + \sqrt{(b^2-4)e^{2r} + 2be^r + 1})^{n+1} - (be^r + 1 - \sqrt{(b^2-4)e^{2r} + 2be^r + 1})^{n+1}}{2^{n+1}\sqrt{(b^2-4)e^{2r} + 2be^r + 1}}.\label{eqn:gnerDef}
\end{align}

Let $\delta_1 = (be^r+ 1 - \sqrt{(b^2-4)e^{2r} + 2be^{t/\sigma_n} + 1})/2$. For large $n$, we have $e^r = 1 + o(1)$. We can write $\delta_1$ as $(b+1-\sqrt{(b+1)^2 - 4})/2 + o(1)$. Now, let $\delta_1' =(b+ 1 + \sqrt{(b+1)^2-4})/2$. Notice $\delta_1\delta_1' = 1 + o(1)$. Since $\delta_1' > 1$, it is clear that $0<\delta_1 < 1$. Hence $\delta_1^{n+1}$ is $o(1)$, and therefore
\begin{align}
g_n(e^r) \ = \  & \frac{\left (\frac{b}2 e^r + \frac12 + \frac12 \sqrt{(b^2-4)e^{2r} + 2be^r + 1} \right)^{n+1} + o(1)}{\sqrt{(b^2-4)e^{2r} + 2be^r + 1}}.
\end{align}
To focus on individual parts of this equation, we define
\begin{align}
\beta_1(r) \: = \  &\ \sqrt{b^2+2 b e^{-r}+e^{-2 r}-4}\label{eqn:beta1}\\
\beta_2(r) \: = \  &\ \frac{b}{2}+\frac{e^{-r}}{2} + \frac{\beta_1(r)}2 \label{eqn:beta2}.
\end{align}
We can now write $g_n(y)$ as
\begin{align}
g_n(e^r) \ = \  \frac{ e^{rn} \beta_2(r)^{n+1}}{\beta_1(r)} + o(1).
\end{align}
We are evaluating $M_{Y_n}(t)$ and require the value of $\log(\beta_2(r))$. We can expand $\beta_2$ and find
\begin{align}
\beta_2(r) \ = \  & \  \frac{1}{2} \left(\sqrt{b^2+2 b-3}+b+1\right)+\left(\frac{-b-1}{2 \sqrt{b^2+2 b-3}}-\frac{1}{2}\right) r \nonumber\\
& +\left(\frac{b^3+3 b^2-b-7}{4 \left(b^2+2 b-3\right)^{3/2}}+\frac{1}{4}\right) r^2+O(r^3)\nonumber\\
\log(\beta_2(r)) \ = \  & \: \log \left(\frac{1}{2} \left(\sqrt{b^2+2 b-3}+b+1\right)\right)-\frac{r}{\sqrt{b^2+2 b-3}} \nonumber\\
 & +\frac{\left(b^2+b-4\right) r^2}{2 \left(b^2+2 b-3\right)^{3/2}}+O\left(r^3\right).
\end{align}
From Lemma \ref{lem:momgenfun}, we have
\begin{align}
M_{Y_n}(t) \ = \  & \  \frac{g_n(e^{t/\sigma_n})e^{-t\mu_n/\sigma_n}}{g_n(1)}\nonumber\\
\log(M_{Y_n}(t)) \ = \  & \: \log(g_n(e^{t/\sigma_n})) -t\mu_n/\sigma_n - \log(g_n(1)).
\end{align}
Because $e^{rn} \beta_2(r)^{n+1}\beta_1(r)^{-1} > 0$, we may move the error term in $g_n(e^r)$ outside the logarithm and simplify:
\begin{align}
\log(M_{Y_n}(t)) \ = \  & \  tn/\sigma_n + (n+1) \log(\beta_2(t/\sigma_n)) - \log(\beta_1(t/\sigma_n)) - t\mu_n/\sigma_n - \log(g_n(1)) + o(1)\nonumber\\
\ = \  & \  tn/\sigma_n + n \log(\beta_2(t/\sigma_n)) + \log(\beta_2(0) + o(1)) - \log(\beta_1(0) + o(1))\nonumber\\
& \  - t\mu_n/\sigma_n - \log(g_n(1)) + o(1).
\end{align}
Plugging in our values of $\beta_1(r)$ from equation \eqref{eqn:beta1}, $\beta_2(r)$ from equation \eqref{eqn:beta2}, $\mu_n$ from Proposition \ref{prop:ValueOfMean}, $\sigma_n$ from Proposition \ref{prop:ValueOfVariance}, $g_n(1)$ from equation \eqref{eqn:gnerDef}, and recalling $r = t/\sigma_n$, we get
\begin{align}
\log(M_{Y_n}(t)) \ = \  & \frac{t^2}2 + O\left(n\left(\frac{t}{\sigma_n}\right)^3\right) \ = \  \frac{t^2}2 + o(1).\end{align}

The moment generating function for a normal distribution is $e^{t\mu + \frac12 \sigma^2t^2}$. Thus, $M_{Y_n}(t)$ pointwise converges to the moment generating function of the standard normal distribution as $n\to\infty$, which from standard probability machinery implies the densities converge to a standard normal.
\end{proof}

\section{Conclusion and Future Questions}
\label{section:conclusion}

Encoding notions of legal decomposition as functions provides a new approach to decomposition problems. We were able to generalize Zeckendorf's (and other) theorems to new classes of sequences which were not amenable to previous techniques. Our work leads to further questions that we hope to return to at a later date. These include:
\begin{enumerate}
\item Our functions $f$ encode notions of legal decomposition where the forbidden terms associated with any $a_n$ are contiguous on the left of $a_n$. Are there weaker conditions on the notion of legal decomposition under which it is possible to construct a sequence $\{a_n\}$ so that all positive integers have unique decompositions using $\{a_n\}$?

\item The distribution of gaps between summands for Zeckendorf decompositions were studied in \cite{BBGILMT}. What is the distribution of gaps (i.e., the difference in indices,  $n_j - n_{j+1}$ in decomposition of $x=\sum_{i=1}^k a_{n_i}$, where $\{n_i\}$ is a decreasing sequence) for the Factorial Number System and $b$-bin decompositions? Can this be studied in general for all $f$-decompositions?

\item The $b$-bin decompositions are examples of $f$-decompositions with periodic functions $f$. Is it true for all periodic $f$ that the number of summands in $f$-decompositions tend to a normal distribution for integers picked from an appropriate growing interval? Under what conditions on $f$ does Gaussian behavior occur?

\item One could attempt to prove results about the number of summands by considering Markov processes where the transition probabilities are related to the $f$-legal decompositions. Such an approach quickly leads to a concentration result for the number of summands (this follows from standard stationarity results), but not to Gaussian behavior. As this method does not weigh all numbers uniformly, however, we do not pursue those investigations here.

\end{enumerate}


\appendix

\section{Example of Linear Recurrence}\label{appendix:exampleoflinearrec}

Theorem \ref{thm:mainfseqrecurrence} uses an algorithm to prove that the $f$-sequences associated to periodic functions satisfy linear recurrence relations; we go through that algorithm for an example below.

Consider the case of 3-bin decompositions. We have a periodic function $f:\N_0\to\N_0$ defined by
\begin{align}
f(n)\ =\ \begin{cases}
1, & \text{if } n \equiv 0 \bmod 3\\
1, & \text{if } n \equiv 1 \bmod 3\\
2, & \text{if } n \equiv 2 \bmod 3.
\end{cases}
\end{align}
The associated $f$-sequence is
\begin{equation}  \{a_n\}_{n=0}^\infty = \{1, 2, 3, 4, 7, 11, 15, 26, 41, 56, 97, 153, 209, 362, 571, 780, 1351, 2131, 2911, \dots \}.\end{equation}
The subsequences $\{a_{i,n}\}_{n=0}^\infty = \{a_{3n + i}\}_{n=0}^\infty$ for $i \in \{0,1,2\}$ are
\begin{align}
\{a_{0, n}\} \ = \  & \{ 1, 4, 15, 56, 209, 780, 2911, \dots \} \\
\{a_{1, n}\} \ = \  & \{ 2, 7, 26, 97, 362, 1351, 5042, \dots \} \\
\{a_{2, n}\} \ = \  & \{ 3, 11, 41, 153, 571, 2131, 7953, \dots \}.
\end{align}

\subsection{Subsequence $\{a_{0,n}\}$}

For any $n \equiv 0 \bmod 3$, we have
\begin{align}
a_n \ = \  & \  a_{n-1} + a_{n-3}\nonumber\\
a_{n-1} \ = \  & \  a_{n-2} + a_{n-3}\nonumber\\
a_{n-3} \ = \  & \  a_{n-4} + a_{n-5}.
\end{align}
These expressions follow a periodic pattern. Their vector representation is
\begin{align}
\begin{array}{r rrr rrr rrr r l}
& * & & & * & & & * & & &* \\
\vec v_0  \ = \  \lbrack & 1 & -1 & 0 & -1  & 0  & 0 & 0 & 0 & 0 & 0 & \rbrack\ \\
\vec v_1 \ = \  \lbrack & 0 &  1 & -1 & -1 & 0  & 0 & 0 & 0 & 0 & 0 & \rbrack\ \\
\vec v_2 \ = \  \lbrack & 0 & 0  &  1 & -1 & -1 & 0 & 0 & 0 & 0 & 0 & \rbrack\ \\
\vec v_3  \ = \  \lbrack & 0 & 0 & 0 & 1 & -1 & 0 & -1  & 0  & 0 & 0 & \rbrack\ \\
\vec v_4 \ = \  \lbrack & 0 & 0 & 0 & 0 &  1 & -1 & -1 & 0  & 0 & 0 & \rbrack\ \\
\vec v_5 \ = \  \lbrack & 0 & 0 & 0 & 0 & 0  &  1 & -1 & -1 & 0 & 0 & \rbrack. \\
\end{array}
\end{align}

Recall from Theorem \ref{thm:mainfseqrecurrence} that we 0-index coordinates of vectors in this algorithm (thus the first coordinate has index 0, the second has index 1 and so on). The stars ($*$) indicate indices that are multiples of $b=3$. All non-zero coordinates that do not fall under these indices are considered ``bad coordinates.''

We begin with $\vec w_0 = \vec v_0 = [ 1 \ -1 \ \ 0 \ -1  \ \ 0  \ \ 0 \ \ 0 \ \ 0 \ \ 0 \ \ 0]$. We use the vectors $\vec u_i$ to keep track of the coordinates between those whose indices are multiples of $b$. In $\vec w_0$, such coordinates are only between indices 0 and $b=3$ (exclusive). Thus $\vec u_0 = [ -1 \ \ 0 ]$.

Using other vectors $\vec v_i$, we move the bad coordinates in $\vec w_0$ to the right, so that the indices between 0 and $b=3$ (exclusive) are all zero. This yields $\vec w_1 = \vec v_0 + \vec v_1 + \vec v_2 = [ 1 \ \ 0 \ \ 0 \ -3  \ -1  \ \ 0 \ \ 0 \ \ 0 \ \ 0 \ \ 0]$. The bad coordinates are only between $b$ and $2b$ (exclusive). Hence $\vec u_1 = [ -1 \ \ 0]$.

We move the bad coordinates to the right once more: $$\vec w_2\ =\ \vec w_1 + \vec v_4 + \vec v_5\ =\ [ 1 \ \ 0 \ \ 0 \ -3  \ \  0 \ \ 0 \ \ -2 \ \ -1 \ \ 0 \ \ 0 ].$$ Here we have $\vec u_2 = [ -1 \ \ 0 ]$.

We have three vectors ($\vec u_0$, $\vec u_1$, and $\vec u_2$) of dimension 2. Therefore, there exists a non-trivial solution to $\sum_{i=0}^2 \lambda_i \vec u_i = 0$. One such solution is $\lambda_0 = -1$, $\lambda_1 = 1$, $\lambda_3 = 0$.

Shifting the $\vec w_i$ vectors using $T_b:\R^{10}\to\R^{10}$ so that the bad coordinates line up, we obtain

\begin{align}
\begin{array}{r rrr rrr rrr r l}
& * & & & * & & & * & & & * &  \\
T_b^2(\vec w_0) \ = \  \lbrack & 0 & 0 & 0 & 0  & 0 & 0 & 1 & -1 & 0 & -1 &  \rbrack \\
T_b(\vec w_1) \ = \  \lbrack & 0 & 0 & 0 & 1  & 0 & 0 & -3 & -1 &  0 & 0 & \rbrack\\
\vec w_2  \ = \  \lbrack & 1 & 0 & 0 & -3 & 0 & 0 & -2 & -1 & 0 & 0 & \rbrack. \\
\end{array}
\end{align}

Now $\sum_{i=0}^b \lambda_i T_b^{b-1-i}(\vec w_i) = - T_b^2 (\vec w_0) + T_b(\vec w_1) = [ 0 \ \ 0 \ \ 0 \ \ 1 \ \ 0 \ \ 0 \ -4 \ \ 0 \ \ 0 \ \ 1]$. Thus for all $n \equiv 0 \bmod 3$, we have $a_n = 4a_{n-3} - a_{n-6}$. The subsequence $\{a_{0, n}\}$ satisfies the recurrence relation $a_{0,n} = 4a_{0,n-1} - a_{0, n-2}$.

\subsection{Common recurrence relation}

Following the same algorithm for $n\equiv 1 \bmod 3$ and $n\equiv 2 \bmod 3$, we find the same relations (i.e., $a_{1,n} = 4a_{1,n-1} - a_{1, n-2}$ and $a_{2,n} = 4a_{2,n-1} - a_{2, n-2}$) as we did for $n\equiv 0 \bmod 3$.

In this case, we do not require Lemma \ref{lem:interlacedsequences} to find a common recurrence relation because we have found the same recurrence relation for all subsequences.

Since all subsequences satisfy $s_n = 4s_{n-1} - s_{n-2}$, the interlaced sequence $\{a_n\}$ (i.e., the $f$-sequence) satisfies $s_n = 4s_{n-3} - s_{n-6}$.

\section{Negative Coefficients in Linear Recurrence}\label{appendix:3binnegativecoeffients}

In \S\ref{sec:zeckbbindecomp} we claimed that the $f$-sequence associated to 3-bin decompositions does not satisfy any linear homogeneous recurrence relation with constant non-negative coefficients. We now give the proof.

First, we prove that each linearly recurrent sequence has a ``minimal linear recurrence relation''. This is a key step in proving that 3-bin decompositions fall outside the scope of previously studied sequences with legal decompositions (the Positive Linear Recurrence Relation Decompositions of \cite{MW1,MW2} or the $G$-ary Representations in \cite{St}).

\begin{lem} For any linearly recurrent sequence $\{a_n\}_{n=0}^\infty$, there exists a linear recurrence relation $\sum_{i=0}^k c_i s_{n-i} = 0$ so that the characteristic polynomial of any linear recurrence relation that $\{a_n\}$ satisfies is multiple of the characteristic polynomial of $\sum_{i=0}^k c_i s_{n-i}$.  \label{lem:minimalrecurrence}
\end{lem}

\begin{proof}
Let $\{s_n\}$ be a placeholder sequence. Let $\mathcal{R}$ be the set of all recurrence relations that $\{a_n\}$ satisfies. Let $L:\mathcal{R}\to\N$ be a function, where $L(r)$ is the degree of the recurrence relation $r$ for any $r\in \mathcal{R}$. By the Well-Ordering Principle, $L(\mathcal{R})$ contains its minimum. Let $k = \min L(\mathcal{R})$ and let $\sum_{i=0}^k c_i s_{n-i}=0$ be a recurrence relation in $\mathcal{R}$ of degree $k$. Notice $k>0$.

Consider any recurrence relation $\sum_{i=0}^l p_i s_{n-i}$ in $\mathcal{R}$. Let $p(x) = \sum_{i=0}^l p_i x^{l-i}$ be the characteristic polynomial of this recurrence relation. Let $c(x) = \sum_{i=0}^k c_i x^{k-i}$ be the characteristic polynomial of $\sum_{i=0}^k c_i s_{n-i}=0$. By the division algorithm for polynomials, there exist polynomials $q(x), r(x)$, where $r(x)$ has degree less than $k$, so that $p(x) = q(x)c(x)+r(x)$.

Assume for contradiction that $r(x) \not \equiv 0$. We know $\{a_n\}$ satisfies the recurrence relation whose characteristic polynomial is $p(x)$. The recurrence relation whose characteristic polynomial is $q(x)c(x)$ is a linear combination of index-shifted versions of $\sum_{i=0}^k c_i s_{n-i}=0$ and hence $\{a_n\}$ satisfies the recurrence relation whose characteristic polynomial is $q(x)c(x)$. Since $\{a_n\}$ satisfies both these recurrence relations, it has to satisfy their difference, whose characteristic polynomial is $r(x)$. However, $r(x)$ has degree less than $k$, which contradicts that $\min L(\mathcal{R}) = k$.
\end{proof}

We are now ready to prove our claim about 3-bin decompositions.

\begin{prop}
The $f$-sequence associated with 3-bin decompositions satisfies no linear homogeneous recurrence relation with non-negative integer coefficients.
\end{prop}
\begin{proof}
        The characteristic polynomial of $s_n = 4s_{n-3} - s_{n-6}$ is $x^6 - 4x^3 + 1$, which is irreducible in $\Q[x]$. Hence the minimal linear recurrence relation for the 3-bin sequence (see Appendix \ref{appendix:exampleoflinearrec}) is $s_n = 4s_{n-3} - s_{n-6}$. We now need to show that no polynomial multiple of $x^6 - 4x^3 + 1$ can be written as $x^h - \sum_{i=1}^h d_i x^{h-i}$ where all $d_i$ are non-negative integers.

Consider any multiple $\sum_{i=0}^{k+6} c_i x^i = (\sum_{j=0}^k p_j x^j)(x^6-4x^3+1)$, where $p_k \neq 0$. This corresponds to a linear recurrence with non-negative coefficients if and only if $c_i \leq 0$ for all $i<k+6$ and $c_{k+6} > 0$. Assume for contradiction that $c_{k+6} > 0$ and $c_i \leq 0$ for all $i < k+6$.

By expanding, we find
\begin{align}
\sum_{i=0}^{k+6} c_i x^i \ = \   \  \left (\sum_{j=0}^k p_j x^j \right ) \left (x^6-4x^3+1\right) \ = \   \sum_{i=0}^{k+6} (p_i - 4p_{i-3} + p_{i-6})x^i
\end{align}

Let $t$ be the smallest non-negative integer so that $p_t \neq 0$.

We claim that for all $n\in\N$ with $t+3n < k+6$, we have $p_{t+3n} \leq 3 p_{t+3n-3}$ and $p_{t+3n} < 0$. The proof follows by induction. In the arguments below we frequently use  $c_{k+6} > 0$ and $c_i \leq 0$ for all $i < k+6$ (which we are assuming to show a contradiction follows).

We have $c_t = p_t$ because $p_i = 0$ for all $i < t$. Since $t\leq k$, we know $p_t = c_t < 0$. Hence $p_t \leq p_{t-3}$. Further, $c_{t+3} = p_{t+3} - 4p_{t} + p_{t-3}$. Since $c_{t+3} \leq 0$ and $p_{t-3} = 0$, we have $p_{t+3} \leq 4p_t < 0$. This proves the base cases $n=0$ and $n=1$.

For any $n$ such that $t+3n < k+6$, we know $c_{t+3n} \leq 0$. This gives us $p_{t+3n} - 4p_{t+3n-3} + p_{t+3n-6} \leq 0$, from which we have $p_{t+3n} \leq 4p_{t+3n-3} - p_{t+3n-6}$. We know $p_{t+3n-3} - p_{t+3n-6} \leq 0$ by the induction hypothesis and hence $p_{t+3n} \leq 3 p_{t+3n-3} < 0$.

By induction, we have $p_{t+3n} \leq 3 p_{t+3n-3}$ and $p_{t+3n} < 0$ for all $n\in\N$ with $t+3n < k+6$.

Choose $n'$ so that $k < t+3n' < k+6$. By the above claim, $p_{t+3n'} < 0$. However, we know $t+3n'>k$ and $p_i = 0$ for all $i>k$. This is a contradiction. Hence the recurrence relation corresponding to any polynomial multiple of $x^6 - 4x^3 + 1$ has at least one negative coefficient $d_i$ when written as $s_n = \sum_{i=1}^{k} d_i s_{n-i}$.
\end{proof}


\ \\

\end{document}